\documentclass{article}
\usepackage[english]{babel} 
\makeatletter
\@ifclassloaded{elsarticle}{%
\newcommand{\ifelsarticle}[2]{#1}%
}{%
\newcommand{\ifelsarticle}[2]{#2}%
}
\makeatother

\ifelsarticle{}{%
\usepackage[cm]{fullpage}
}

\usepackage{amssymb}    
\usepackage{float}     
\usepackage{epic}       
\usepackage{fancybox}   
\usepackage{graphicx}
\usepackage{amsmath}
\usepackage{amsthm} 
\usepackage{tikz}
\usepackage{pgfplots}
\usepackage{amsfonts}
\usepackage{hyperref} 

\usetikzlibrary{calc,trees,fadings,fit,positioning,arrows,shapes.symbols,decorations.markings,shapes.geometric}

\newcommand{\Q}{\mathbb{Q}}
\newcommand{\N}{\mathbb{N}}
\newcommand{\Z}{\mathbb{Z}}
\newcommand{\Znn}{\Z^{n\times n}}

\newtheorem{proposition}{Proposition}
\newtheorem{definition}[proposition]{Definition}
\newtheorem{notation}[proposition]{Notation}
\newtheorem{theorem}[proposition]{Theorem}

\newtheorem{lemma}[proposition]{Lemma}

\catcode`\¿=\active\def¿{?`} 

\renewcommand*{\thefootnote}{\dag}
\begin{document}

\ifelsarticle{%
\begin{frontmatter}
}{}
\title{Classes of Symmetric Cayley Graphs over Finite Abelian Groups of Degrees 4 and 6\footnote{A previous version of some of the results in this paper were first announced at the 2010 International Workshop on Optimal Interconnection Networks (IWONT 2010)}}
\author{Crist\'obal Camarero, Carmen Mart\'inez and Ram\'on Beivide}
\ifelsarticle{%
\address{Electronics and Computers Department. University of Cantabria}
}

\ifelsarticle{}{\maketitle}

\begin{abstract}

The present work is devoted to characterize the family of symmetric undirected Cayley graphs over finite Abelian groups for degrees 4 and 6.

\end{abstract}

\ifelsarticle{%

\end{frontmatter}
}{}

\section{Introduction}

The Cayley graph over the group $\Gamma$ and generating set $S\subset \Gamma$, denoted by $Cay(\Gamma;S)$, is defined as the graph with vertex set $\Gamma$ and with set of adjacencies
$\{(x,xg)\mid x\in\Gamma,\ g\in S\}$. If $S = -S$, then the graph is undirected.
The present work is devoted to characterize the symmetric members of the family of undirected Cayley graphs over finite Abelian groups for degrees 4 and 6.
Since these graphs are known to be vertex-transitive \cite{Akers}, the characterization will be done by determining those being edge-transitive.\par

In this paper the matricial notation by Fiol \cite{Fiol} for Cayley graphs over finite Abelian groups will be used. Hence, in order to be self-contained let us establish in this section the notation, definitions and results that will be used in this paper.

\begin{notation} The following notation will be used throughout the article:
\begin{itemize}
\item Lower case letters denote integers: $a$, $b$, $\dotsc$
\item Bold font denotes integer column vectors: $\mathbf{v}$, $\mathbf{w}$, $\dotsc$
\item Capitals correspond to integral matrices: $M$, $P$, $\dotsc$
\item $\mathbf{e}_i$ denotes the vector with a $1$ in its $i$-th component and $0$ elsewhere.
\item $\mathcal{B}_{n} = \{\mathbf{e}_i \ | \ i=1, \dotsc, n\}$ denotes the $n$-dimensional orthonormal basis. Then, $\pm \mathcal{B}_{n} = \{\pm \mathbf{e}_i \ | \ i=1, \dotsc, n\}$.
\end{itemize}
\end{notation}

\begin{definition} \cite{Fiol} Let $M \in\mathbb{Z}^{n\times n}$
be a non-singular square matrix of dimension $n$. Two vectors $\mathbf{v}, \mathbf{w} \in \mathbb{Z}^{n}$ are congruent modulo $M$ if and only if we have
$\mathbf u = \begin{pmatrix} u_1\\ u_2\\ \vdots \\ u_n \\ \end{pmatrix} \in \mathbb{Z}^{n}$ such that: $\mathbf{v}-\mathbf{w} = u_1 \mathbf{m}_1 + u_2 \mathbf{m}_2 + \dotsb +u_n\mathbf{m}_n = M\mathbf u,$
where $\mathbf{m}_{j}$ denotes the $j$-th column of $M$. We will denote this congruence as $\mathbf{v} \equiv \mathbf{w} \pmod{M}$.
\end{definition}

Given a square non-singular integral matrix $M \in \mathbb{Z}^{n \times n}$, we will consider the Cayley graph $Cay(\mathbb{Z}^{n}/M\mathbb{Z}^{n}; \pm \mathcal{B}_{n})$, hence:

\begin{itemize}
\item The vertex set is $\mathbb{Z}^{n}/M\mathbb{Z}^{n} = \{\mathbf{v} \pmod{M} \ | \ \mathbf{v} \in \mathbb{Z}^{n}\}$.
\item Two vertices $\mathbf{v}$ and $\mathbf{w}$ are adjacent if and only if $\mathbf{v}-\mathbf{w} \equiv \pm \mathbf e_i \pmod{M}$ for some $i\in\{1, \dotsc, n\}$.
\end{itemize}

From here onwards, all matrices will be considered to be non-singular, unless the contrary is stated. Note that, since $\mathbb{Z}^{n}/M\mathbb{Z}^{n}$ has $|\det (M)|$ elements, this will be the number of nodes of $Cay(\mathbb{Z}^{n}/M\mathbb{Z}^{n}; \pm \mathcal{B}_{n})$. Moreover, since any vertex $\mathbf{v}$ is adjacent to $\mathbf{v} \pm \mathbf{e}_i \pmod{M}$, $Cay(\mathbb{Z}^{n}/M\mathbb{Z}^{n}; \pm \mathcal{B}_{n})$ is, in general, a regular graph of degree $2n$. Then, we call $n$ the \textsl{dimension} of the graph. Note that when $\mathbf e_i\equiv \pm \mathbf e_j\pmod M$ or $2\mathbf e_i\equiv \mathbf 0\pmod M$ for some $1\leq i,j\leq n$ then the degree of the Cayley graph is less than $2n$. In that case we can also consider the corresponding multigraph, which always has degree $2n$.
The hypercube could be considered as an extreme case since $\forall i \in \{1, \ldots , n\}, 2\mathbf e_i\equiv \mathbf 0\pmod M$ and therefore it has degree $n$.\par

The following result shows that considering $Cay(\mathbb{Z}^{n}/M\mathbb{Z}^{n}; \pm \mathcal{B}_{n})$ does not imply any loss of generality, that is,
any Cayley graph over a finite Abelian group is isomorphic to $Cay(\mathbb{Z}^{n}/M\mathbb{Z}^{n}; \pm \mathcal{B}_{n})$, for some matrix $M \in \mathbb{Z}^{n\times n}$.

\begin{theorem} For any connected Cayley graph $G$ over a finite Abelian group there is $M \in \mathbb{M}^{n \times n}$ non-singular such that $G \cong Cay(\mathbb{Z}^{n}/M\mathbb{Z}^{n}; \pm \mathcal{B}_{n}).$
\end{theorem}

\begin{proof}
Let the Cayley graph $G=Cay(\Gamma;\{g_1,\dotsc,g_n\})$ with $\Gamma$ Abelian and finite. We proceed by induction.
If $n=1$ then $M=\begin{pmatrix}o(g_1)\end{pmatrix}$.
Otherwise, let $M_{n-1}$ be such that $Cay(\mathbb{Z}/M_{n-1}\mathbb{Z}; \pm \mathcal{B}_{n-1})\cong Cay(\Gamma;\{g_1,\dotsc,g_{n-1}\})$
with an isomorphism $f(\mathbf e_i)=g_i$. Then, let $a$ be the minimum positive integer such that $ag_n=x_1g_1+x_2g_2+\dotsb +x_{n-1}g_{n+1}$ for integers $x_i$ (which exists because $\Gamma$ is finite).
Then $M=\begin{pmatrix}M_{n-1}&\mathbf x\\0&a\end{pmatrix}$ satisfies $Cay(\mathbb{Z}^{n}/M\mathbb{Z}^{n}; \pm \mathcal{B}_{n}) \cong G$.
\end{proof}

Hence, we will denote $Cay(\mathbb{Z}^{n}/M\mathbb{Z}^{n}; \pm \mathcal{B}_{n})$ by $\mathcal{G}(M)$. In \cite{Fiol} it was shown how performing different operations on matrix $M$
remains the same graph. The next definitions and results recall this fact.

\begin{definition} \label{def:rightequivalent} Let $M_1, M_2 \in \mathbb{Z}^{n \times n}$. Then, $M_1$ is \textsl{right equivalent} to $M_2$, denoted by $M_1 \cong M_2$, if and only if there exists a unit matrix $P \in \mathbb{Z}^{n \times n}$ such that $M_1 = M_2 P$.
\end{definition}

\begin{theorem}\label{theo:isoright} \cite{Fiol}. If matrices a pair of matrices $M_1, M_2 \in \mathbb{Z}^{n \times n}$ are right-equivalent, then $\mathcal{G}(M_1)$ and $\mathcal{G}(M_2)$ are isomorphic graphs.
\end{theorem}

%
%

\begin{definition}A \textsl{signed permutation matrix} is a matrix with entries in
$\{-1,0,1\}$ which has exactly one $\pm 1$ in each row and column.\end{definition}

Note that  in $\Znn$ the signed permutation matrices are exactly the unitary matrices, this is,
the matrices $U$ such $UU^t=I$. They are related to permutations in the way that for each permutation $\sigma\in \Sigma_n$ there is a unique signed permutation matrix $P_\sigma$ such that
$$P_\sigma\begin{pmatrix}v_1\\\vdots\\v_n\end{pmatrix}=
\begin{pmatrix}\pm v_{\sigma(1)}\\\vdots\\\pm v_{\sigma(n)}\end{pmatrix}.$$

\begin{theorem}\label{theo:isoleft}\cite{Fiol} If $P$ is a signed permutation matrix then $\mathcal G(PM)\cong \mathcal G(M)$.
\end{theorem}

In this paper we will find matrices such that $\mathcal{G}(M)$ is edge-transitive for dimensions 2 and 3. In the first case, the characterization will be complete, that 
is we will find all $\mathcal{G} (M)$ being symmetric. In the case of dimension 3, we will only consider those being edge-transitive by means of linear automorphism, as explained later.

%

With this aim, in Section \ref{sec:automorfismos} we will consider some properties of
isomorphisms between Cayley graphs over finite Abelian groups, their automorphisms and
the implications of containing cycles of length 4. In Section \ref{sec:simetricos}, we give some general results for the characterization of $\mathcal{G} (M)$ graphs of any dimension symmetric by means of linear automorphisms.
In Section \ref{sec:dimension2} we give the full characterization of symmetric $\mathcal{G} (M)$ graphs of dimension 2 (degree 4).
In Section \ref{sec:dimension3} we give the full characterization of symmetric by linear automorphisms $\mathcal{G} (M)$ graphs of dimension 3 (degree 6).

\section{Linear Automorphisms of Cayley Graphs $\mathcal{G}(M)$ and 4-cycles}\label{sec:automorfismos}

Given a graph $G = (V, E)$, $Aut(G)$ denotes its automorphisms group. $G$ is said to be \textsl{vertex-transitive} if, for
any pair of vertices $v_1, v_2 \in V$ there exists $f \in Aut(G)$ such that $f(v_1) = v_2$. Similarly, $G$ is said to be
\textsl{edge-transitive} if for any pair of edges $e_1 = (v_1, v_2) , e_2 \in E$ there exists $f \in Aut(G)$ such that $f(e_1)= (f(v_1), f(v_2)) =
e_2$. Then, if $G$ is both vertex and edge transitive, then it is called \textsl{symmetric}. The subgroup of $Aut(G)$ of elements which fix some element $x\in V$ is denoted as $Aut(G,x)$ (also known as \textsl{stabilizer}).\par

All Cayley graphs graphs are vertex-transitive \cite{Akers}. The linear automorphisms of a Cayley graph $\mathcal G(M)$ form a group $LAut(\mathcal{G}(M))$.
This group usually coincides with the full automorphism group $Aut(\mathcal{G}(M))$, except in a few cases that we consider separately.
We also consider the group of linear automorphisms which fixes $\mathbf 0$, $LAut(\mathcal{G}(M),\mathbf 0)$.

\begin{definition}
$\mathcal{G}(M)$ is said \textsl{linearly edge-transitive} if for every $i$
there exists $f\in LAut(\mathcal{G}(M),\mathbf 0)$ such that $f(\mathbf e_1)=\pm \mathbf e_i$.
\end{definition}

Clearly, a linearly edge-transitive Cayley graph $\mathcal{G}(M)$ is symmetric. Therefore, in this section we study the automorphism group of $\mathcal{G}(M)$ graphs. A basic question is determining when there are nonlinear automorphisms; which is very related to the problem of determining \textsl{\'{A}d\'{a}m isomorphy} \cite{Adam,Delorme}. A pair of graphs $\mathcal{G}(M_1)$ and $\mathcal{G}(M_2)$ are \textsl{\'{A}d\'{a}m isomorphic} if there exists an isomorphism between their groups of vertices such that it sends the set of generators of one graph into the generators of the other graph. It is clear that any \'{A}d\'{a}m isomorphic
graphs are isomorphic, but the opposite it is not always true.\par

In \cite{Delorme} it was proved that any pair of isomorphic Cayley multigraphs of degree four are \'{A}d\'{a}m isomorphic unless the pair is (up to \'{A}d\'{a}m isomorphy)
$\begin{pmatrix}2k+1&2\\1&2\end{pmatrix},\begin{pmatrix}2k&2\\0&2\end{pmatrix}$ for some integer $k$. Using the nonlinear isomorphism between them one can build a nonlinear automorphism in each; hence they will appear in our study of the nonlinear automorphisms of dimension 2. However the reverse is not true, as there are a few graphs with nonlinear automorphisms which do not have a pairing non-\'{A}d\'{a}m isomorphic graph.

\begin{definition} The \textsl{neighborhood} of a vertex $v$
in the graph $G = (V,E)$ is defined as $N(v)=\{w\mid (v,w)\in E\}$. Then, the
\textsl{common neighborhood} of a list of vertices $v_1,\dotsc,v_n\in V$
as $N(v_1,\dots,v_n)=\bigcap_{i=1}^n N(v_i)$.
\end{definition}

\begin{theorem}\label{theo:neighborhoodpreserved} The neighborhood is preserved in graph isomorphisms.
That is, if $f$ is a graph isomorphism, then
$$N(f(v_1),\dots,f(v_n))=\{f(w)\mid w\in N(v_1,\dots,v_n)\}.$$
\end{theorem}

\begin{proof} Let $f$ be a graph isomorphism from $G = (V,E)$ into $G' =
(V',E')$.  We have that $f(w)\in N(f(v_1),\dots,f(v_n))$ if only if $\forall
i,f(w)\in N(f(v_i))$, that is $\forall i,(f(w),f(v_i))\in E'$. Since $f$ is an
isomorphism we have that this is equivalent to $\forall i,(w,v_i)\in E$ so
$w\in N(v_1,\dots,v_n)$.
\end{proof}

Next, we analyze which isomorphisms between Cayley graphs over finite Abelian groups are linear mappings. This is related to the following concept.

\renewcommand*{\thefootnote}{\ddag}
\begin{definition} We say that $\mathbf a, \mathbf b, \mathbf c, \mathbf d\in \pm\mathcal B_n$ form a \emph{4-cycle} in
$\mathcal{G}(M)$ if  $\mathbf 0\equiv\mathbf a+\mathbf b+\mathbf c+\mathbf d\pmod M$\footnote{each of $\{(\mathbf v,\mathbf v+\mathbf a,\mathbf v+\mathbf a+\mathbf b,\mathbf v+\mathbf a+\mathbf b+\mathbf c,\mathbf v+\mathbf a+\mathbf b+\mathbf c+\mathbf d)\mid \mathbf v\in \mathbb Z^n/M\mathbb Z^n)\}$ is a cycle of length 4.}.
If we have $\mathbf a\in\{-\mathbf b,-\mathbf c,-\mathbf d\}$ then we call the cycle \emph{trivial}.
Then, we say that $\mathcal{G}(M)$ has not nontrivial 4-cycles if all its 4-cycles are trivial.
\end{definition}

\begin{theorem}\label{theo:neighborhoodadjacency} If $\mathcal{G}(M)$ is has not
nontrivial 4-cycles, then for all $\mathbf a,\mathbf b\in \mathcal \pm \mathcal B_n$ with $\mathbf a\neq \mathbf b$
	$$N(\mathbf a,\mathbf b)=\{\mathbf 0,\mathbf a+\mathbf b\}.$$
\end{theorem}

\begin{proof}
If $\mathbf v\in N(\mathbf a,\mathbf b)$ then $\exists \mathbf x,\mathbf y\in \pm\mathcal B_n$ such that $\mathbf v=\mathbf a+\mathbf x=\mathbf b+\mathbf y$.
Since we have $\mathbf a-\mathbf b+\mathbf x-\mathbf y=\mathbf 0$ and $\mathcal{G}(M)$ has not nontrivial 4-cycles, it must be fulfilled one of the following expressions:
\begin{itemize}
    \item $\mathbf a=\mathbf b$ contradicting the hypothesis,
    \item $\mathbf a=-\mathbf x$ and thus $\mathbf v=\mathbf a-\mathbf a=\mathbf 0$,
    \item $\mathbf a=\mathbf y$ and thus $\mathbf v=\mathbf b+\mathbf y=\mathbf a+\mathbf b$.
\end{itemize}
\end{proof}

\begin{lemma}\label{lem:translatedautomorphism} If $f$ is an automorphism of $\mathcal{G}(M)$, then
for any $\mathbf t\in \Z^n/M\Z^n$, $f_{\mathbf t}:\mathbf x\mapsto f(\mathbf t+\mathbf x)-f(\mathbf t)$ is an automorphism of $\mathcal{G}(M)$ with $f_{\mathbf t}(\mathbf 0)=\mathbf 0$.
\end{lemma}
\begin{proof}We have $f_{\mathbf t}(\mathbf 0)=f(\mathbf t+\mathbf 0)-f(\mathbf t)=\mathbf 0$, thus $f_{\mathbf t}$ fixes $\mathbf 0$.
Now if $\mathbf x\in \Z^n/M\Z^n$ is adjacent to $\mathbf y\in \Z^n/M\Z^n$ then $\mathbf t+\mathbf x$ is adjacent to $\mathbf t+\mathbf y$ and then as $f$ is an automorphism we have that $f(\mathbf t+\mathbf x)$ is adjacent to $f(\mathbf t+\mathbf y)$. Hence $f_{\mathbf t}(\mathbf x)$ is adjacent to $f_{\mathbf t}(\mathbf y)$.
\end{proof}

\begin{lemma}\label{lem:aristacommfull}
Let $\mathcal{G}(M)$ be such that it has not nontrivial 4-cycles.
Then for any $f\in Aut(\mathcal{G}(M),\mathbf 0)$
we have that $f(\mathbf a+\mathbf b)=f(\mathbf a)+f(\mathbf b)$ for any $\mathbf a,\mathbf b\in \pm \mathcal B_n$.
\end{lemma}

\begin{proof} Let $\mathbf a,\mathbf b\in \pm\mathcal B_n$. First we prove the lemma for $\mathbf a\neq \mathbf b$.
From Theorem \ref{theo:neighborhoodadjacency} we get that
$N(\mathbf a,\mathbf b)=\{\mathbf 0,\mathbf a+\mathbf b\}$, hence by Theorem \ref{theo:neighborhoodpreserved} $N(f(\mathbf a),f(\mathbf b))=\{f(\mathbf 0),f(\mathbf a+\mathbf b)\}=\{\mathbf 0,f(\mathbf a)+f(\mathbf b)\}$.
As $f(\mathbf 0)=\mathbf 0$ we have that $f(\mathbf a+\mathbf b)=f(\mathbf a)+f(\mathbf b)$.\par

Now note that since for any $\mathbf a\in \pm\mathcal B_n$, $\mathbf a\neq -\mathbf a$ we have that for any $f\in Aut(\mathcal{G}(M),\mathbf 0)$, $f(-\mathbf a)=-f(\mathbf a)$.\par

It remains to prove that $f(2\mathbf a)=2f(\mathbf a)$. Consider the automorphism $f'$ defined by $f'(\mathbf v)=f(\mathbf a+\mathbf v)-f(\mathbf a)$ (it is an automorphism by Lemma \ref{lem:translatedautomorphism}).
We have $f'(-\mathbf a)=-f'(\mathbf a)$, hence $f(\mathbf a-\mathbf a)-f(\mathbf a)=-(f(\mathbf a+\mathbf a)-f(\mathbf a))$. Rearranging terms we obtain the desired $f(2\mathbf a)=2f(\mathbf a)$.
\end{proof}

\begin{lemma}\label{lem:commlinear} If $\forall \mathbf a,\mathbf b\in\pm\mathcal B_n,\ f\in Aut(\mathcal{G}(M),\mathbf 0),\ f(\mathbf a+\mathbf b)=f(\mathbf a)+f(\mathbf b)$ then
every $f\in Aut(\mathcal{G}(M),\mathbf 0)$ is a group automorphism of $\Z^n/M\Z^n$.
\end{lemma}

\begin{proof}
First we prove that for all $f \in Aut\bigl(\mathcal{G}(M),\mathbf 0\bigr)$ we have that
$$\forall \mathbf t\in \mathcal{G}(M),\ f(\mathbf t+\mathbf a+\mathbf b)=f(\mathbf t+\mathbf a)+f(\mathbf t+\mathbf b)-f(\mathbf t).$$
Let $\mathbf t \in \mathcal{G}(M)$. We define $f_{\mathbf t}(\mathbf v)=f(\mathbf t+\mathbf v)-f(\mathbf t)$, by Lemma \ref{lem:translatedautomorphism} $f_{\mathbf t}\in Aut(\mathcal{G}(M),\mathbf 0)$.
By hypothesis, we have $\forall \mathbf t\in \mathcal{G}(M),\ f_{\mathbf t}(\mathbf a+\mathbf b)=f_{\mathbf t}(\mathbf a)+f_{\mathbf t}(\mathbf b)$, which implies
$\forall \mathbf t\in \mathcal{G}(M),\ f(\mathbf t+\mathbf a+\mathbf b)-f(\mathbf t)=f(\mathbf t+\mathbf a)-f(\mathbf t)+f(\mathbf t+\mathbf b)-f(\mathbf t)$.\par

We need to prove $\forall n_i\in\N,\ f(\sum_i n_i\mathbf e_i)=\sum_in_if(\mathbf e_i)$.
We proceed by induction in $N=\sum_i n_i$; for $N=0,1$ it is immediate.
Now let $\mathbf v=\sum_i n_i\mathbf e_i$ and $\sum_i n_i=N+1$. Let $u,v$ be any positive integers such that $n_u+n_v\geq 2$.
Now, because of the first claim,
$f(\mathbf v)=
	f((\mathbf v-\mathbf e_u-\mathbf e_v)+\mathbf e_u+\mathbf e_v)=
	f((\mathbf v-\mathbf e_u-\mathbf e_v)+\mathbf e_u)+f((\mathbf v-\mathbf e_u-\mathbf e_v)+\mathbf e_v)-f(\mathbf v-\mathbf e_u-\mathbf e_v)$.
	Applying the induction hypothesis we have that:
$f(\mathbf v)=
	\bigl(f(\mathbf v-\mathbf e_u-\mathbf e_v)+f(\mathbf e_u)\bigr)
	+\bigl(f(\mathbf v-\mathbf e_u-\mathbf e_v)+f(\mathbf e_v)\bigr)
	-f(\mathbf v-\mathbf e_u-\mathbf e_v)=
	f(\mathbf v-\mathbf e_u-\mathbf e_v)+f(\mathbf e_u)+f(\mathbf e_v)
	$.
Then as $f(\mathbf v-\mathbf e_u-\mathbf e_v)=
	f(\sum_i n_i\mathbf e_i -\mathbf e_u-\mathbf e_v)=
	\sum_i n_if(\mathbf e_i) -f(\mathbf e_u)-f(\mathbf e_v)$
	we have that $f(\mathbf v)=\sum_i n_if(\mathbf e_i)$.
\end{proof}

\begin{theorem}\label{theo:linear} If the graph $\mathcal{G}(M)$ has not nontrivial 4-cycles then any graph automorphism with $f(\mathbf 0)=\mathbf 0$ is a group automorphism of $\Z^n/M\Z^n$.
\end{theorem}
\begin{proof}
If there are not nontrivial 4-cycles then by Lemma \ref{lem:aristacommfull} we have
for any $f\in Aut(\mathcal{G}(M),\mathbf 0)$
that $f(\mathbf a+\mathbf b)=f(\mathbf a)+f(\mathbf b)$ for any $\mathbf a,\mathbf b\in \pm\mathcal B_n$.
Now we have linearity by Lemma \ref{lem:commlinear}.
\end{proof}

\section{Edge-transitivity of Cayley graphs $\mathcal{G}(M)$ by Linear Automorphisms}\label{sec:simetricos}


In this section we will consider those graphs $\mathcal{G}(M)$ such that any of its automorphisms is a linear mapping.

\begin{theorem}\label{theo:linealtopermutation}
For any $f\in LAut(\mathcal{G}(M),\mathbf 0)$ there exists a signed permutation matrix $P$ such that $\forall \mathbf a\in \mathbb{Z}^n / M \mathbb{Z}^n ,\ f(\mathbf a)=P\mathbf a$.
\end{theorem}

\begin{proof}
We define $P$ as:
$$P_{i,j}=\left\{
    \begin{array}{rl}
    1&\text{if }f(\mathbf e_j)=\mathbf e_i\\
    -1&\text{if }f(\mathbf e_j)=-\mathbf e_i\\
    0&\mbox{otherwise}
    \end{array}\right.$$
having $f(\mathbf e_i)=P\mathbf e_i$.
	Let $\mathbf a=\sum n_i\mathbf e_i$.
$$f(\mathbf a)=\sum_i n_if(\mathbf e_i)=\sum_i n_iP\mathbf e_i=P\sum_i n_i\mathbf e_i=P\mathbf a$$
\end{proof}

\begin{theorem}\label{theo:PM=MQ}For any $M \in \mathbb{Z}^{n \times n}$ the mapping $f(\mathbf x)=P\mathbf x$ is a linear automorphism of $\mathcal{G}(M)$ if only if there exists $Q \in \mathbb{Z}^{n \times n}$ such that $PM=MQ$.
\end{theorem}

\begin{proof}
We prove first the left to right implication.
	As $f$ must be well-defined, for all $i$, $\mathbf 0=P\mathbf 0\equiv PM\mathbf e_i\pmod M$.
	And then exists $\mathbf q_i$ such that $PM\mathbf e_i=M\mathbf q_i$, gathering all $i$s
	together
	$$PM=[PM\mathbf e_1,\dots,PM\mathbf e_n]=M[\mathbf q_1,\dots,\mathbf q_n]=MQ.$$

For the right to left implication; by Theorem \ref{theo:isoleft} $f$ is an isomorphism from $\mathcal G(M)$ into $\mathcal G(PM)=\mathcal G(MQ)$. Then by Theorem \ref{theo:isoright} $f$ is an automorphism of $\mathcal G(M)$.
\end{proof}

To know if $\mathcal{G}(M)$ is linearly edge-transitive we need to look to the multiplicative group of the signed permutation matrices $P$ such $PM=MQ$.
It is clear that, if a matrix representing a cycle of length $n$
(even if it changes signs) is in the group then by composing it with itself,
we can map $\mathbf e_1$ to every $\mathbf e_i$ making the graph edge-transitive.
In these cases we have that $LAut(\mathcal{G}(M),\mathbf 0)$ is a cyclic group.
The smallest dimension for which we found $LAut(\mathcal{G}(M),\mathbf 0)$ to be noncyclic
is for $n=4$ with the Klein four-group. That situation occurs for example for Lipschitz graphs, which were introduced in in \cite{IEEETITQuat}.
Since we just consider dimensions 2 and 3, this will not suppose any problem.\par

\begin{definition} Two matrices $A,B\in\mathbb{\Z}^{n\times n}$ are \textsl{similar}, denoted
by $A\sim B$, if there exists a unit matrix $U \in\mathbb{\Z}^{n\times n}$ such that $AU=UB$. 
\end{definition}

\begin{lemma}\label{lem:rightsimilar}
Let $PM=MQ$ and $PM'=M'Q'$. Then, $M\cong M'$ if and only if $Q\sim Q'.$
\end{lemma}
\begin{proof}
Since $PM=MQ$ and $M=M'U$ then $PM'U=M'UQ$ and $PM'=M'(UQU^{-1})=M'Q'$
with $Q'\sim Q$. Reciprocally, we know that if $PM=MQ$ and $Q'=UQU^{-1}$
then $M'=MU$ produces $PM'=M'Q'$ and $M'\cong M$.
\end{proof}


Since right equivalences leave the group invariant (Theorem \ref{theo:isoright}),
we know that for a given $P$ we only need to see how many $Q$ there are
modulo similarity. Then, knowing $P$ and $Q$ we can
solve for $M$. In \cite{Newman} the next theorems are stated, which will be very helpful in the determination of $Q$ in the following Sections \ref{sec:dimension2} and \ref{sec:dimension3}.

\begin{theorem}[\cite{Newman}, Theorem III.12, page 50]\label{theorem:newmanreducible}
Given a matrix $A$ we can find a similar matrix, made of blocks,
which is block upper triangular and moreover, that the blocks of the diagonal
all have characteristic polynomial irreducible over $\Q$ .
\end{theorem}

\begin{theorem}[\cite{Newman}, Theorem III.14, pag 53, The theorem of Lattimer and MacDuffee] \label{theorem:newmanirreducible}
If we have a matrix with irreducible characteristic polynomial, like
the produced by the previous theorem then the number of matrices modulo
similarity is the class number of $\Z[\theta]$ where $\theta$ is a
root of the polynomial.
\end{theorem}

\section{Characterization of Symmetric $\mathcal{G} (M)$ Graphs of Dimension 2}\label{sec:dimension2}

 The complete characterization of symmetric $\mathcal{G} (M)$ graphs with $M \in \mathbb{Z}^{2 \times 2}$ will be done in this section. Firstly, we will consider those which 
 are edge-transitive by means of linear automorphism. Later, we will consider those cases involving non-linear automorphisms. 
 
By Theorem \ref{theo:linealtopermutation} a graph $\mathcal{G}(M)$ is linearly edge-transitive if there is an automorphism $f$ with $f(\mathbf e_1)=\pm \mathbf e_2$ and $f(\mathbf e_2)=\pm \mathbf e_1$. Such automorphism is associated to one of the matrices:
	$\begin{pmatrix}0&1\\1&0\end{pmatrix}$,
	$\begin{pmatrix}0&-1\\1&0\end{pmatrix}$,
	$\begin{pmatrix}0&1\\-1&0\end{pmatrix}$,
	$\begin{pmatrix}0&-1\\-1&0\end{pmatrix}$.

Since $\begin{pmatrix}0&1\\-1&0\end{pmatrix}^3=\begin{pmatrix}0&-1\\1&0\end{pmatrix}$ there is only need to check $\begin{pmatrix}0&-1\\1&0\end{pmatrix}$ and $\pm \begin{pmatrix}0&1\\1&0\end{pmatrix}$.

\begin{figure}
    \begin{center}
		\includegraphics[width=.3\columnwidth]{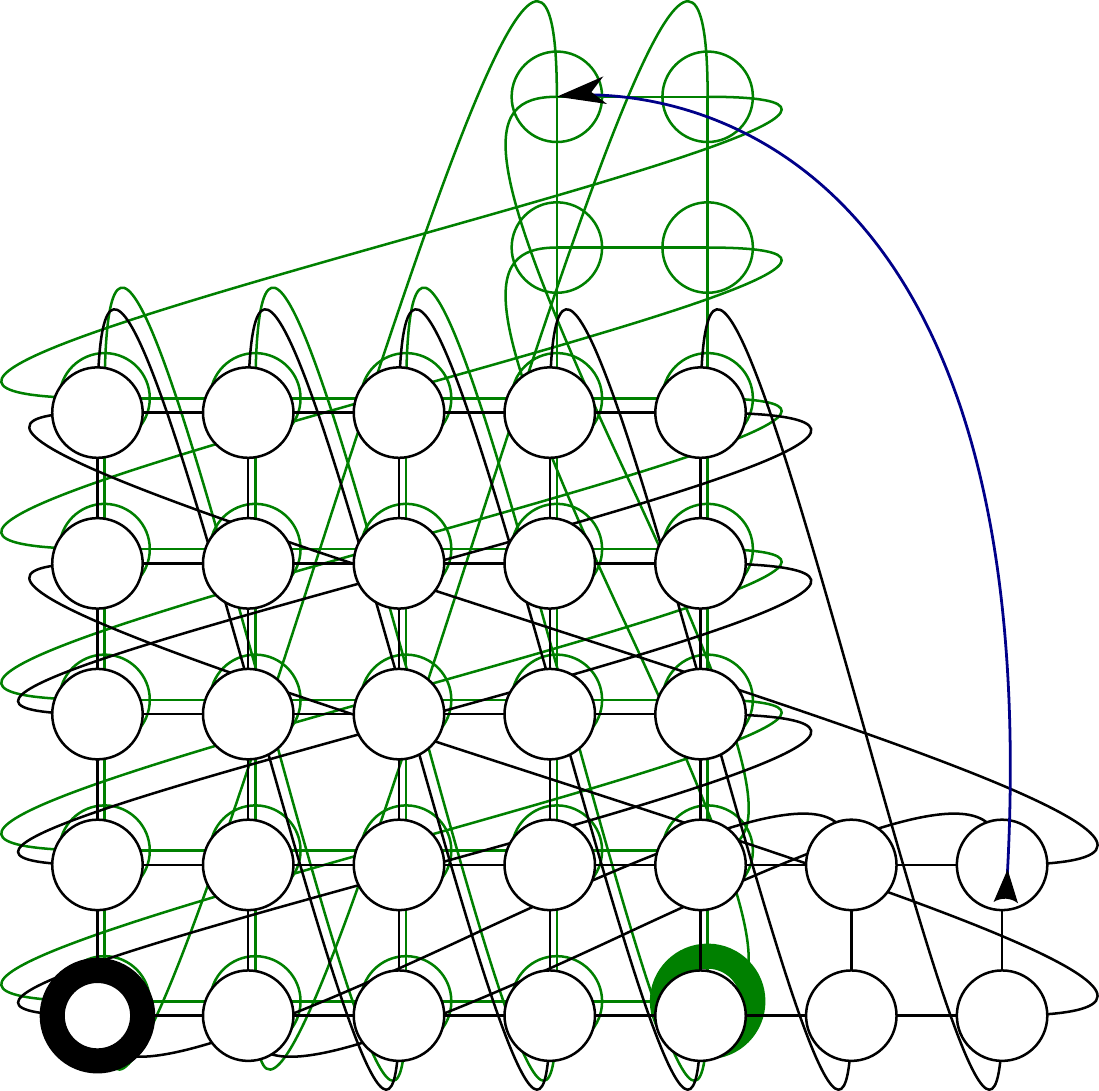}
	    \includegraphics[width=.3\columnwidth]{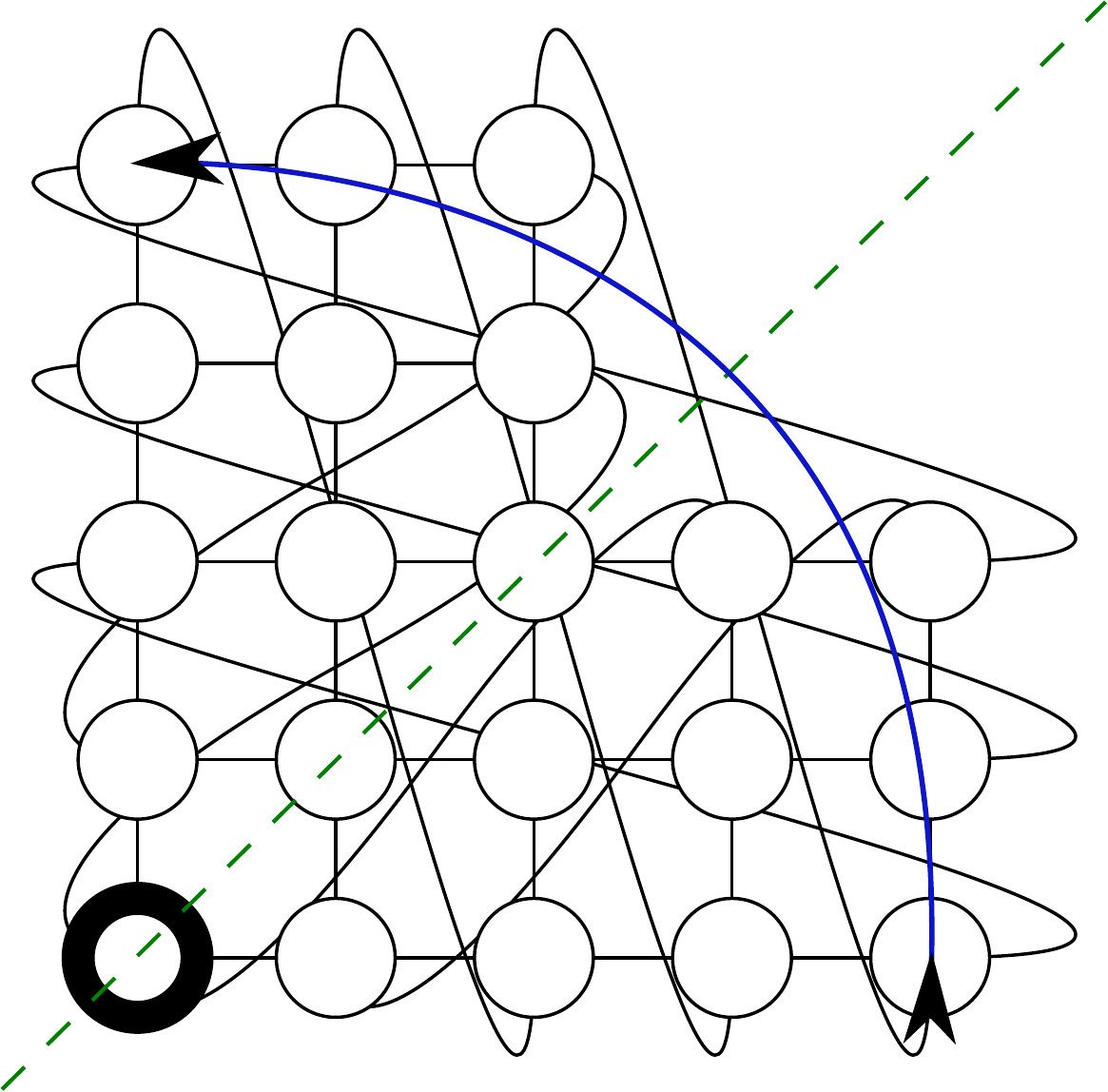}
    \end{center}
   	\caption{Linear Automorphisms of Cayley Graphs of Dimension 2.}
    \label{fig:automorphisms}
\end{figure}

\begin{theorem} \label{theo:simetricos}Let $M \in \mathbb{Z}^{2 \times 2}$ be non-singular. Then, $\mathcal{G}(M)$ is linearly edge-transitive if and only if, for some $a,b\in\mathbb Z, $
$M$ is right equivalent to one of the following matrices:

$$ M_1 = \begin{pmatrix}a& b\\b&a\end{pmatrix},\ M_2 = \begin{pmatrix}a& -b\\b&a\end{pmatrix},\ M_3 = \begin{pmatrix}a& -b\\a&b\end{pmatrix}.$$

\end{theorem}

\begin{proof} Let $M=\begin{pmatrix}a&b\\c&d\end{pmatrix}$. We will determine $Q$ and solve the system $PM=MQ$; which by Theorem \ref{theo:PM=MQ} is a necessary and sufficient condition to be linearly edge-transitive.
The characteristic polynomial of $\pm \begin{pmatrix}0&1\\1&0\end{pmatrix}$ is $\lambda^2-1$, and the one of  $\begin{pmatrix}0&-1\\1&0\end{pmatrix}$ is $\lambda^2+1$.
As $PM=MQ$ it must be the characteristic polynomial of both $P$ and $Q$.
By Lemma \ref{lem:rightsimilar} we can choose any matrix similar to $Q$ and obtain a matrix right-equivalent to $M$.
Therefore, we have two cases:
\begin{itemize}
\item $P=\pm \begin{pmatrix}0&1\\1&0\end{pmatrix}$,
    $\lambda^2-1=(\lambda+1)(\lambda-1)$,
    being reducible over $\Q$, by Theorem \ref{theorem:newmanreducible} $Q$ must be similar
    to a matrix $Q'=\begin{pmatrix}1&p\\0&-1\end{pmatrix}$, which is similar
    to either $\begin{pmatrix}1&0\\0&-1\end{pmatrix}$ or to
    $\begin{pmatrix}1&1\\0&-1\end{pmatrix}\sim\begin{pmatrix}0&1\\1&0\end{pmatrix}$.
	In the first case, depending on $P$ we obtain $M$ equal to $\begin{pmatrix}a& b\\b&a\end{pmatrix}$ or to $\begin{pmatrix}a& b\\-b&-a\end{pmatrix}\cong \begin{pmatrix}a& -b\\-b&a\end{pmatrix}$; which are the same under the variable change $b\mapsto -b$.
	In the second case, the same happens for the possible matrices $\begin{pmatrix}a& -b\\a&b\end{pmatrix}$ and $\begin{pmatrix}a& b\\-a&b\end{pmatrix}\cong \begin{pmatrix}b& -a\\b&a\end{pmatrix}$ and the variable change $a\mapsto b$, $b\mapsto a$.
\item $P=\begin{pmatrix}0&-1\\1&0\end{pmatrix}$,
    $\lambda^2+1$ which is irreducible over $\Q$ and the class number
    of $\Z[i]$ is 1, so by Theorem \ref{theorem:newmanirreducible} $Q$ must be similar to $P$.
	The only possible solutions are $\begin{pmatrix}a& -b\\b&a\end{pmatrix}$.
\end{itemize}
\end{proof}

The first two cases of Theorem \ref{theo:simetricos}, $\mathcal G(M_1)$ and $\mathcal G(M_2)$, are depicted in Figure \ref{fig:automorphisms}.
As it was proved in \cite{ISIT10}, $\mathcal G(M_1)$ and $\mathcal G(M_3)$ are isomorphic to the Kronecker product of two cycles. Furthermore $\mathcal G(M_2)$ is isomorphic to the Gaussian graph introduced in \cite{IEEE_TC}.

\subsection{Edge-transitive $\mathcal{G}(M)$ Graphs of Dimension 2 by Nonlinear Automorphisms}\label{sec: no lineales}

In this subsection we focus on those $\mathcal{G}(M)$ graphs with nontrivial 4-cycles, and hence, according to Theorem \ref{theo:linear} their group of automorphisms could contain nonlinear automorphisms.
Clearly, if there is a nontrivial 4-cycle then there exist $\mathbf a,\mathbf b\in \pm\mathcal B_n$ which fulfill:

\begin{enumerate}
\item \label{caso4a} $4\mathbf a \equiv \mathbf 0 \pmod{M}$
\item \label{caso3ab} $3\mathbf a+\mathbf b \equiv \mathbf 0 \pmod{M}$
\item \label{caso2a2b}$2\mathbf a+2\mathbf b \equiv \mathbf 0 \pmod{M}$
\end{enumerate}

If we consider $u\mathbf a+v\mathbf b\equiv \mathbf 0 \pmod{M}$ it means that there exists $\mathbf x \in \mathbb{Z}^{2}$ such that
$\mathbf k=\begin{pmatrix}u\\v\end{pmatrix}=M\mathbf x$. Now, let $\gcd(\mathbf x)=\gcd(x_1,\dotsc,x_n)$,
$\mathbf x'=\frac{\mathbf x}{\gcd\mathbf x}$ and $\mathbf k'=\frac{\mathbf k}{\gcd\mathbf x}$,
having $\mathbf k'=M\mathbf x'$. As $\gcd\mathbf x'=1$ we can build a unit matrix $U$
with $\mathbf x'$ as one of its columns, and therefore $M'=MU$ has $\mathbf k'$ as a column.
In addition, Theorem \ref{theo:isoleft} allows to choose each component positive.\par

We will begin with item (\ref{caso2a2b}). In this case we obtain the matrix $M=\begin{pmatrix}u&2\\v&2\\\end{pmatrix}$. If $v=2k$ we have
that $\begin{pmatrix}u&2\\2k&2\\\end{pmatrix}$ is right equivalent to $\begin{pmatrix}u-v&2\\0&2\\\end{pmatrix}$. On the other hand, if
$v=2k+1$ then $\begin{pmatrix}u&2\\2k+1&2\\\end{pmatrix}$ is right equivalent to $\begin{pmatrix}u-v+1&2\\1&2\\\end{pmatrix}$. Both matrices generate the same graph
and there is a nonlinear isomorphism between them.
In addition note that if the first column has odd weight then $\begin{pmatrix}0&1\\1&0\end{pmatrix}\begin{pmatrix}2k+1&2\\0&2\end{pmatrix}\cong \begin{pmatrix}2k+2&2\\1&2\end{pmatrix}$, so there is a linear isomorphism in addition to the nonlinear one. Hence, the only pairs non \'{A}d\'{a}m isomorphic are $(\begin{pmatrix}2k+1&2\\1&2\end{pmatrix},\begin{pmatrix}2k&2\\0&2\end{pmatrix})$, which correspond with the ones in \cite{Delorme}.
Furthermore, the matrices of these non \'{A}d\'{a}m isomorphic graphs satisfy $\begin{pmatrix}0&1\\1&0\end{pmatrix}M\cong M$, thus by Theorem \ref{theo:PM=MQ} they are actually linearly edge-transitive.

For the items (\ref{caso4a}) and (\ref{caso3ab}) we begin proving that if there is exactly one nontrivial 4-cycle, then all automorphisms are linear. Furthermore note that these results are also valid in any number of dimensions.

\begin{lemma}\label{lem:negateinverse} Let $f\in Aut(\mathcal{G}(M),\mathbf 0)$ and $\mathbf a\in\pm\mathcal B_n$. If $\forall\mathbf b\in\pm\mathcal B_n\setminus\{\mathbf a,-\mathbf a\},\ f(-\mathbf b)=-f(\mathbf b)$
then $f^{-1}(-\mathbf a)=-f^{-1}(\mathbf a)$.
\end{lemma}
\begin{proof}
We check three cases.
\begin{itemize}
\item Case $f^{-1}(\mathbf a)\neq \pm\mathbf a$. Applying the hypothesis we get $\mathbf a=f(f^{-1}(\mathbf a))=-f(-f^{-1}(\mathbf a))$. Then $f^{-1}(-\mathbf a)=f^{-1}(f(-f^{-1}(\mathbf a)))=-f^{-1}(\mathbf a)$.
\item Case $f^{-1}(-\mathbf a)\neq \pm\mathbf a$. Applying the hypothesis we get $-\mathbf a=f(f^{-1}(-\mathbf a))=-f(-f^{-1}(-\mathbf a))$. Then $f^{-1}(\mathbf a)=f^{-1}(f(-f^{-1}(-\mathbf a)))=-f^{-1}(-\mathbf a)$.
\item Case $\{f^{-1}(\mathbf a),f^{-1}(-\mathbf a)\}\subseteq \{\pm \mathbf a\}$.
	As $f^{-1}$ is a bijection we have the equality
	$\{f^{-1}(\mathbf a),f^{-1}(-\mathbf a)\}=\{\pm \mathbf a\}$.
	Now $f^{-1}(\mathbf a)+f^{-1}(-\mathbf a)=\mathbf a+(-\mathbf a)=\mathbf 0$.
\end{itemize}
\end{proof}

\begin{theorem}If the only nontrivial 4-cycle is $4\mathbf a\equiv \mathbf 0\pmod M$ or $3\mathbf a+\mathbf b\equiv \mathbf 0\pmod M$ then
$$Aut(\mathcal{G}(M))=LAut(\mathcal{G}(M)).$$
\end{theorem}
\begin{proof}We proceed proving several claims iteratively.
\begin{enumerate}
\item\label{neighxy} For all $\mathbf x,\mathbf y\in \pm\mathcal B_n\setminus \{\mathbf a,-\mathbf a\}$, $\mathbf x\neq\mathbf y$, $N(\mathbf x,\mathbf y)=\{\mathbf 0,\mathbf x+\mathbf y\}$.

We have $N(\mathbf x,\mathbf y)=\{\mathbf v\mid \mathbf v=\mathbf x+\mathbf p=\mathbf y+\mathbf q,\ \mathbf p,\mathbf q\in\pm\mathcal B_n\}$. That is, we look for 4-cycles $\mathbf x+\mathbf p-\mathbf y-\mathbf q=\mathbf 0$. The trivial ones are $\mathbf x=-\mathbf p$ and $\mathbf x=\mathbf q$ which respectively give $\mathbf v=\mathbf 0$ and $\mathbf v=\mathbf x+\mathbf y$.
	If it is the nontrivial 4-cycle $4\mathbf a=\mathbf 0$ then we have $\{\mathbf x,\mathbf y\}=\{\mathbf a,-\mathbf a\}$, contradicting the hypothesis.
	If it is the nontrivial 4-cycle $3\mathbf a+\mathbf b=\mathbf 0$, then at least one of $\mathbf x$ or $\mathbf y$ is $\pm\mathbf a$.

\item\label{neighax} For all $\mathbf x\in \pm\mathcal B_n\setminus \{\mathbf a,-\mathbf a\}$, $N(\mathbf a,\mathbf x)\subseteq\{\mathbf 0,\mathbf a+\mathbf x,2\mathbf a\}$.

In this case we look for nontrivial 4-cycles $\mathbf a+\mathbf p-\mathbf x-\mathbf q=\mathbf 0$. As $\mathbf x\not\in\{\pm \mathbf a\}$, we have $\mathbf p=-\mathbf q=\mathbf a$ and then $\mathbf v=\mathbf a+\mathbf p=2\mathbf a$. Note that if we have the cycle $4\mathbf a=\mathbf 0$ then we only have the trivial solutions.

\item\label{neighaa} $N(\mathbf a,-\mathbf a)=\{\mathbf 0,\pm 2\mathbf a\}$.

In this case we look for nontrivial 4-cycles $\mathbf a+\mathbf p+\mathbf a-\mathbf q=\mathbf 0$. At least one of $\mathbf p,-\mathbf q$ is equal to $\mathbf a$. If $\mathbf p=\mathbf a$ then $\mathbf v=\mathbf a+\mathbf p=2\mathbf a$. If $-\mathbf q=\mathbf a$ then $\mathbf v=-\mathbf a+\mathbf q=-2\mathbf a$.

\item\label{oppx} For all $\mathbf x\in\pm\mathcal B_n\setminus \{\pm\mathbf a\}$, $f\in Aut(\mathcal{G}(M),\mathbf 0)$, $f(-\mathbf x)=-f(\mathbf x)$.

We have $4\mathbf x\neq \mathbf 0$, since it would be another nontrivial 4-cycle. Hence $\mathbf x\neq -\mathbf x$ and by item (\ref{neighxy}) $N(\mathbf x,-\mathbf x)=\{\mathbf 0\}$. By Theorem \ref{theo:neighborhoodpreserved} we have $N(f(\mathbf x),f(-\mathbf x))=\{\mathbf 0\}$, thus $f(\mathbf x)+f(-\mathbf x)=\mathbf 0$.

\item\label{single} For all $\mathbf x\in\pm\mathcal B_n$, $f\in Aut(\mathcal{G}(M),\mathbf 0)$, $f(-\mathbf x)=-f(\mathbf x)$ and $f(2\mathbf x)=2f(\mathbf x)$.

First apply Lemma \ref{lem:negateinverse} together item (\ref{oppx}) to $f^{-1}$ to get $\forall f\in Aut(\mathcal{G}(M),\mathbf 0),\ \allowbreak f(-\mathbf a)=-f(\mathbf a)$.
Then considering the automorphism $f'$ defined by $f'(\mathbf v)=f(\mathbf x+\mathbf v)-f(\mathbf x)$ like in the proof of Lemma \ref{lem:aristacommfull} we obtain that $f(2\mathbf x)=2f(\mathbf x)$.

\item\label{fa} For all $f\in Aut(\mathcal{G}(M),\mathbf 0)$, $f(\pm\mathbf a)=\pm \mathbf a$.

From item (\ref{neighaa}) we have $N(\mathbf a,-\mathbf a)=\{\mathbf 0,\pm 2\mathbf a\}$ with $\mathbf 0\neq \pm2\mathbf a$. By Theorem \ref{theo:neighborhoodpreserved} we get $N(f(\mathbf a),f(-\mathbf a))=\{\mathbf 0,f(\pm 2\mathbf a)\}$ with $\mathbf 0\neq f(\pm2\mathbf a)$. By items (\ref{neighxy}, \ref{neighax}, \ref{neighaa}) we get $\{f(\mathbf a),f(-\mathbf a)\}=\{\pm\mathbf a\}$.

\item\label{pairs} For all $f\in Aut(\mathcal{G}(M),\mathbf 0)$, $\mathbf x,\mathbf y\in\pm\mathcal B_n$, $f(\mathbf x+\mathbf y)=f(\mathbf x)+f(\mathbf y)$.

If $\mathbf x=\mathbf y$ it is item (\ref{single}). Otherwise if neither of $\mathbf x,\mathbf y$ is in $\{\pm \mathbf a\}$ we proceed like the first step of the proof of Lemma \ref{lem:aristacommfull}; from item (\ref{neighxy}) we get
$N(\mathbf x,\mathbf y)=\{\mathbf 0,\mathbf x+\mathbf y\}$, hence by Theorem \ref{theo:neighborhoodpreserved} $N(f(\mathbf x),f(\mathbf y))=\{f(\mathbf 0),f(\mathbf x+\mathbf y)\}=\{\mathbf 0,f(\mathbf x)+f(\mathbf y)\}$.
As $f(\mathbf 0)=\mathbf 0$ we have that $f(\mathbf x+\mathbf y)=f(\mathbf x)+f(\mathbf y)$.
Now if some is in $\{\pm \mathbf a\}$, we assume without loss of generality that $\mathbf y=\mathbf a$ and $\mathbf x\not\in\{\pm \mathbf a\}$.
From item (\ref{neighax}) we have $N(\mathbf a,\mathbf x)\subseteq\{\mathbf 0,\mathbf a+\mathbf x,2\mathbf a\}$. And by Theorem \ref{theo:neighborhoodpreserved} that $N(f(\mathbf a),f(\mathbf x))\subseteq\{\mathbf 0,f(\mathbf a+\mathbf x),f(2\mathbf a)\}$.
By item (\ref{fa}) we have $N(f(\mathbf a),f(\mathbf x))\subseteq\{\mathbf 0,f(\mathbf a)+f(\mathbf x),2f(\mathbf a)\}$. As $f(2\mathbf a)=2f(\mathbf a)$ (item (\ref{single})) we have that $f(\mathbf a+\mathbf x)=f(\mathbf a)+f(\mathbf x)$.

\item $Aut(\mathcal{G}(M))=LAut(\mathcal{G}(M))$.

Apply Lemma \ref{lem:commlinear} to item (\ref{pairs}).
\end{enumerate}
\end{proof}

Finally, there are a few marginal cases in which the graph contains several nontrivial 4-cycles.
These matrices are the matrices whose both columns correspond to nontrivial 4-cycles and their left divisors.
These matrices can be built by selecting two columns in the set:
$$ C=\left\{
        \begin{pmatrix}4\\0\end{pmatrix},
        \begin{pmatrix}3\\1\end{pmatrix},
        \begin{pmatrix}1\\3\end{pmatrix},
        \begin{pmatrix}0\\4\end{pmatrix},
        \begin{pmatrix}3\\-1\end{pmatrix},
        \begin{pmatrix}1\\-3\end{pmatrix},
        \begin{pmatrix}2\\0\end{pmatrix},
        \begin{pmatrix}0\\2\end{pmatrix}
    \right\}.$$

A complete study of the following cases, shows as that most of the combinations
are edge-transitive. However, there are cases that lack of a nonlinear
automorphism, leading to non-edge-transitive graphs.\par

Up to isomorphism, the bidimensional $\mathcal{G} (M)$ graphs with 2 different nontrivial solutions for
4-cycles are:
	\begin{figure}
		\begin{center}
		\includegraphics[width=.3\textwidth]{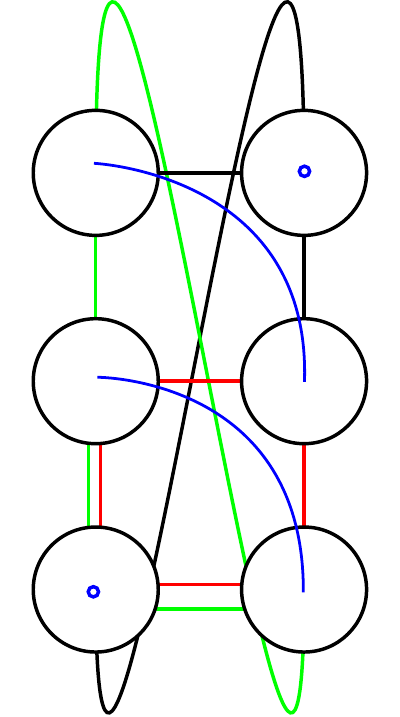}
		\end{center}
		\caption[A nonlinear automorphism in $( (2,0)^t , (-1,3)^t)$]
			{A nonlinear automorphism of $\mathcal{G}(M)$, where
			$M = \begin{pmatrix}2&-1\\0&3\end{pmatrix}$.}
		\label{fig:automnl2103}
	\end{figure}
	\begin{figure}
		\begin{center}
		\includegraphics[width=.5\textwidth]{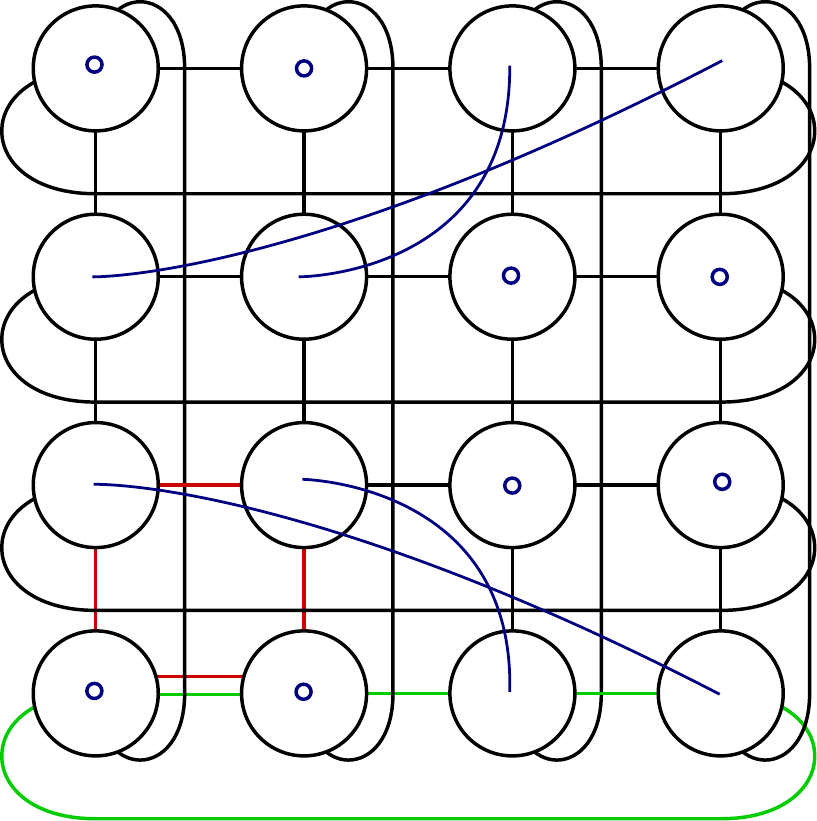}
		\end{center}
		\caption{A nonlinear automorphism of the square torus of side 4.}
		\label{fig:automnl4004}
	\end{figure}
\begin{itemize}
\item With nontrivial 4 cycles but without nonlinear automorphisms.
	$$
	\begin{pmatrix}1&0\\3&2\\\end{pmatrix},
	\begin{pmatrix}2&0\\0&2\\\end{pmatrix},
	\begin{pmatrix}4&3\\0&1\\\end{pmatrix},
	\begin{pmatrix}4&1\\0&3\\\end{pmatrix}$$
\item With a nonlinear automorphism, which makes them edge-transitive,
	$$\begin{pmatrix}4&0\\0&2\\\end{pmatrix},
	\begin{pmatrix}3&3\\1&-1\\\end{pmatrix} \cong
	\begin{pmatrix}2&-1\\0&3\\\end{pmatrix},
	\begin{pmatrix}3&1\\1&2\\\end{pmatrix}$$
	with an example in Figure \ref{fig:automnl2103}.
	The first two have degree 3. Their associated Cayley multigraphs do not have nonlinear automorphisms.
	In the figure, we show in blue a nonlinear automorphism involution, which fixes two vertices and maps the nontrivial green 4-cycle into the red 4-cycle.
\item With a nonlinear automorphism, but their linear automorphisms already make
	them edge-transitive,
	$$\begin{pmatrix}4&0\\0&4\\\end{pmatrix},
	\begin{pmatrix}3&1\\1&3\\\end{pmatrix},
	\begin{pmatrix}3&-1\\1&3\\\end{pmatrix}$$
	with the torus as example in Figure \ref{fig:automnl4004}.
\end{itemize}

\section{Linearly Edge-Transitive $\mathcal{G} (M)$ Graphs of Dimension 3}\label{sec:dimension3}

This section provides a complete characterization of those $\mathcal{G} (M)$ graphs with $M \in \mathbb{Z}^{3 \times 3}$ being linearly edge-transitive.

\begin{lemma}\label{lem:symorder3} Given $M\in\Z^{3\times 3}$, $\mathcal{G}(M)$ is linearly edge-transitive if and only if there exists a signed permutation matrix of order 3 in $LAut(\mathcal{G}(M),\mathbf 0)$.
\end{lemma}

\begin{proof}
If such a signed permutation exists, it is clear that $\mathcal{G}(M)$ is linearly edge-transitive.\par

For the reciprocal, by Theorem \ref{theo:linealtopermutation} the automorphism is a signed permutation matrix.
We can check that signed permutations matrices of dimension 3 can have orders 1, 2, 3, 4 and 6.
The identity is the only signed permutation matrix of order 1 and it does not contribute to symmetry.
Moreover, the signed permutation matrices which only change signs (that is, which are diagonal matrices) do no contribute to symmetry.
Any remaining signed permutation matrix of orders 2 and 4 do not provide symmetry by themselves,
since they fix one of the components,
and the composition of two of them generates either a sign change or a signed permutation matrix of
order 3 or 6.\par

Hence linear edge-transitivity implies the existence of an automorphism $f\in LAut(\mathcal{G}(M),\mathbf 0)$ with order 3 or 6. If it has order 3, we already have the desired matrix.
Otherwise we have $f^3=-id$ and so $g=f^2$ has order 3.
\end{proof}

Hence, if $\mathcal{G}(M)$ is linearly edge-transitive then $LAut(\mathcal{G}(M),\mathbf 0)$
contains at least one of the next four cyclic groups as a subgroup and by Theorem \ref{theo:PM=MQ}
there is a matrix $P$ such that $PM=MQ$ for some $Q$.


\begin{align*}
	P_1&=\begin{pmatrix}0&0&1\\1&0&0\\0&1&0\end{pmatrix}&
	P_2&=\begin{pmatrix}0&0&1\\-1&0&0\\0&-1&0\end{pmatrix}\\
	P_3&=\begin{pmatrix}0&0&-1\\1&0&0\\0&-1&0\end{pmatrix}&
	P_4&=\begin{pmatrix}0&0&-1\\-1&0&0\\0&1&0\end{pmatrix}
\end{align*}

These signed permutation matrices have characteristic and minimum polynomial $\lambda ^3 -1$.
We can find some matrices (symbolic over 3 integer parameters)
whose Cayley graphs are edge-transitive by taking $Q=P$, that is, we obtain $M_i$ such that $P_iM_i=M_iP_i$. They are:

\begin{align*}
M_1&=\begin{pmatrix}a&c&b\\b&a&c\\c&b&a\end{pmatrix},&
M_2&=\begin{pmatrix}a&-c&-b\\b&a&-c\\c&b&a\end{pmatrix},\\
M_3&=\begin{pmatrix}a&-c&-b\\b&a&c\\c&-b&a\end{pmatrix},&
M_4&=\begin{pmatrix}a&c&b\\b&a&-c\\c&-b&a\end{pmatrix}.
\end{align*}

Next, we find the similar matrices.

\begin{lemma}\label{lem:similarclasses}
There are exactly 2 similarity classes with characteristic polynomial $\lambda ^3 -1$:
$$
Q_1=\begin{pmatrix}1&0&0\\0&-1&1\\0&-1&0\end{pmatrix}\text{ and }
Q_2=\begin{pmatrix}1&0&1\\0&-1&1\\0&-1&0\end{pmatrix}.
$$
\end{lemma}

\begin{proof}
For $\lambda^3-1=(\lambda-1)(\lambda(\lambda+1)+1)$ we have the following upper triangular block matrix
which has it as its characteristic polynomial:
$Q=\begin{pmatrix}1&0&0\\0&-1&1\\0&-1&0\end{pmatrix}$.
We know that
$$
\begin{pmatrix}1&v&u\\0&1&0\\0&0&1\end{pmatrix}
\begin{pmatrix}1&u+2v&u-v\\0&-1&1\\0&-1&0\end{pmatrix}
\begin{pmatrix}1&-v&-u\\0&1&0\\0&0&1\end{pmatrix}
=
\begin{pmatrix}1&0&0\\0&-1&1\\0&-1&0\end{pmatrix}$$
So
	$\forall u,v\in\mathbb Z,\
\begin{pmatrix}1&u+2v&u-v\\0&-1&1\\0&-1&0\end{pmatrix}
\sim
\begin{pmatrix}1&0&0\\0&-1&1\\0&-1&0\end{pmatrix}
	$.
Since $|\det(\begin{pmatrix}-1&-2\\-1&1\end{pmatrix})|=3$,
by Theorem \ref{theorem:newmanreducible},
we have at most 3 matrices modulo similarity, which are:
$$
\begin{pmatrix}1&0&0\\0&-1&1\\0&-1&0\end{pmatrix}\text{, }
\begin{pmatrix}1&0&1\\0&-1&1\\0&-1&0\end{pmatrix}\text{ and }
\begin{pmatrix}1&0&2\\0&-1&1\\0&-1&0\end{pmatrix}.
$$
We check that the first two are non-similar.
If
$$
\begin{pmatrix}1&0&0\\0&-1&1\\0&-1&0\end{pmatrix}
\begin{pmatrix}a&b&c\\d&e&f\\g&h&i\end{pmatrix}
=
\begin{pmatrix}a&b&c\\d&e&f\\g&h&i\end{pmatrix}
\begin{pmatrix}1&0&1\\0&-1&1\\0&-1&0\end{pmatrix}
$$
then
$$
\begin{pmatrix}a&b&c\\-d+g&-e+h&-f+i\\-d&-e&-f\end{pmatrix}
=
\begin{pmatrix}a&-b-c&a+b\\d&-e-f&d+e\\g&-h-i&g+h\end{pmatrix}.
$$
Hence $d=g=0$ and $a=-3b$; and $3b$ divides the determinant, which cannot be a unit.
Now we see that the last two are similar.
$$
\begin{pmatrix}1&0&1\\0&-1&1\\0&-1&0\end{pmatrix}
\begin{pmatrix}1&0&1\\0&0&1\\0&-1&1\end{pmatrix}
=
\begin{pmatrix}1&0&1\\0&0&1\\0&-1&1\end{pmatrix}
\begin{pmatrix}1&0&2\\0&-1&1\\0&-1&0\end{pmatrix}
$$
So we have proved that there are exactly 2 similarity classes
with characteristic polynomial $\lambda^3-1$:
$$
Q_1=\begin{pmatrix}1&0&0\\0&-1&1\\0&-1&0\end{pmatrix}\text{ and }
Q_2=\begin{pmatrix}1&0&1\\0&-1&1\\0&-1&0\end{pmatrix}.
$$
\end{proof}

Finally, we explore the $4\cdot 2=8$ possible matrices from all the combinations.

\begin{lemma}\label{lem:similarities} With the previous definitions, $P_1\sim Q_2\sim P_2\sim P_3 \sim P_4$.
\end{lemma}
\begin{proof}
First we see that $P_1\sim Q_2$.
$$
\begin{pmatrix}0&0&1\\1&0&0\\0&1&0\end{pmatrix}
\begin{pmatrix}1&0&0\\1&-1&1\\1&0&1\end{pmatrix}
=
\begin{pmatrix}1&0&0\\1&-1&1\\1&0&1\end{pmatrix}
\begin{pmatrix}1&0&1\\0&-1&1\\0&-1&0\end{pmatrix}
$$
And now that $P_1\sim P_2\sim P_3 \sim P_4$.
$$
\begin{pmatrix}0&0&1\\1&0&0\\0&1&0\end{pmatrix}
\begin{pmatrix}-1&0&0\\0&1&0\\0&0&-1\end{pmatrix}
=
\begin{pmatrix}-1&0&0\\0&1&0\\0&0&-1\end{pmatrix}
\begin{pmatrix}0&0&1\\-1&0&0\\0&-1&0\end{pmatrix}
$$
$$
\begin{pmatrix}0&0&1\\1&0&0\\0&1&0\end{pmatrix}
\begin{pmatrix}1&0&0\\0&1&0\\0&0&-1\end{pmatrix}
=
\begin{pmatrix}1&0&0\\0&1&0\\0&0&-1\end{pmatrix}
\begin{pmatrix}0&0&-1\\1&0&0\\0&-1&0\end{pmatrix}
$$
$$
\begin{pmatrix}0&0&1\\1&0&0\\0&1&0\end{pmatrix}
\begin{pmatrix}1&0&0\\0&-1&0\\0&0&-1\end{pmatrix}
=
\begin{pmatrix}1&0&0\\0&-1&0\\0&0&-1\end{pmatrix}
\begin{pmatrix}0&0&-1\\-1&0&0\\0&1&0\end{pmatrix}
$$
\end{proof}
Thus, the first 4 matrices with $P_iM=MQ_2$ are right equivalent to the
previously calculated $M_i$. Therefore, we find the 4 symbolic matrices $M_i'$ which satisfy $P_iM_i'=M_i'Q_1$.

\begin{align*}
M_1'&=\begin{pmatrix}a&b&c\\a&c&-b-c\\a&-b-c&b\end{pmatrix}&
M_2'&=\begin{pmatrix}a&b&c\\-a&-c&b+c\\a&-b-c&b\end{pmatrix}\\
M_3'&=\begin{pmatrix}a&b&c\\a&c&-b-c\\-a&b+c&-b\end{pmatrix}&
M_4'&=\begin{pmatrix}a&b&c\\-a&-c&b+c\\-a&b+c&-b\end{pmatrix}
\end{align*}

Now we have all the necessary elements to enunciate the tridimensional characterization of linearly edge-transitive graphs.

\begin{theorem} Let $M \in \mathbb{Z}^{3 \times 3}$ be non-singular. Then, $\mathcal{G}(M)$ is linearly edge-transitive  if and only if it is isomorphic to $\mathcal{G}(M_1)$ or $\mathcal{G}(M_1')$, where:
$$
M_1=\begin{pmatrix}a&c&b\\b&a&c\\c&b&a\end{pmatrix}\text{ or }
M_1'=\begin{pmatrix}a&b&c\\a&c&-b-c\\a&-b-c&b\end{pmatrix}
$$
for some $a,b,c\in\mathbb Z$.
\end{theorem}
\begin{proof}Let $\mathcal{G}(M)$ be linearly edge-transitive with $M\in\mathbb Z^{3\times 3}$.
By Lemma \ref{lem:symorder3}, $P$ must exist with $PM=MQ$ with $P\in\{P_1,P_2,P_3,P_4\}$.
By Lemmas \ref{lem:rightsimilar} and \ref{lem:similarclasses} there exist $M'$ and $Q$
with $M\cong M'$, $Q\in\{Q_1,Q_2\}$ and $PM'=M'Q$.

\begin{itemize}
\item  If $Q=Q_2$, then by Lemma \ref{lem:similarities} we know $M''\in\{M_1,M_2,M_3,M_4\}$,
with $PM''=M''P$, $M''\cong M$.
Now we want to see that the matrices $M_1$, $M_2$, $M_3$ and $M_4$ generate the same set
of matrices modulo graph-isomorphism.
For each $M_i$ we find a variable change and isomorphism from $M_1$ into $M_i$:

$$
\begin{pmatrix}-1&0&0\\0&1&0\\0&0&1\end{pmatrix}
M_1
\begin{pmatrix}1&0&0\\0&-1&0\\0&0&-1\end{pmatrix}
=
\begin{pmatrix}-a&c&b\\b&-a&-c\\c&-b&-a\end{pmatrix}
$$
which is $M_4$ giving $a$ the value $-a$.

$$
\begin{pmatrix}1&0&0\\0&-1&0\\0&0&1\end{pmatrix}
M_1
\begin{pmatrix}-1&0&0\\0&1&0\\0&0&-1\end{pmatrix}
=
\begin{pmatrix}-a&c&-b\\b&-a&c\\-c&b&-a\end{pmatrix}
$$
which is $M_2$ giving $a$ the value $-a$ and $c$ the value $-c$.

$$
\begin{pmatrix}1&0&0\\0&1&0\\0&0&-1\end{pmatrix}
M_1
\begin{pmatrix}1&0&0\\0&1&0\\0&0&-1\end{pmatrix}
=
\begin{pmatrix}a&c&-b\\b&a&-c\\-c&-b&a\end{pmatrix}
$$
which is $M_3$ giving $c$ the value $-c$.

\item If $Q=Q_1$, then by Lemma \ref{lem:similarities} we know $M'\in\{M_1',M_2',M_3',M_4'\}$.
Now we want to see that the matrices $M_1'$, $M_2'$, $M_3'$ and $M_4'$ generate the same set
of matrices modulo graph-isomorphism.
For each $M_i$ we find an isomorphism from $M_1$ into $M_i$; we do not need in this case variable changes:
$$
M_1'
=\begin{pmatrix}1&0&0\\0&-1&0\\0&0&1\end{pmatrix}M_2'
=\begin{pmatrix}1&0&0\\0&1&0\\0&0&-1\end{pmatrix}M_3'
=\begin{pmatrix}1&0&0\\0&-1&0\\0&0&-1\end{pmatrix}M_4'
$$
\end{itemize}
\end{proof}

\ifelsarticle{%
\bibliographystyle{elsarticle-num}
\section*{References}
}{%
\bibliographystyle{plain}
}
\bibliography{main}

\end{document}